\documentclass[12pt]{article}
\usepackage{amsfonts}

\usepackage{graphicx}
\usepackage{amsmath}

\newtheorem{theorem}{Theorem}

\newtheorem{definition}[theorem]{Definition}

\newtheorem{lemma}[theorem]{Lemma}

\newtheorem{proposition}[theorem]{Proposition}
\newtheorem{remark}[theorem]{Remark}

\newenvironment{proof}[1][Proof]{\textbf{#1.} }{\ \rule{0.5em}{0.5em}}

\begin{document}

\title{Picard iterations for diffusions on symmetric matrices}
\author{ 
Carlos G. Pacheco\thanks{Departamento de Matematicas, CINVESTAV-IPN, A. Postal 14-740, Mexico D.F. 07000, MEXICO. Email: cpacheco@math.cinvestav.mx}
}

%\date{}
\maketitle

\begin{abstract}
Matrix-valued stochastic processes have been of significant importance in areas such as physics, engineering and mathematical finance. 
One of the first models studied has been the so-called Wishart process, which is described as the solution of a stochastic differential equation in the space of matrices.
In this paper we analyze natural extensions of this model, and prove the existence and uniqueness of the solution.
We do this by carrying out a Picard iteration technique in the space of symmetric matrices. 
This approach takes into account the operator character of the matrices, which helps to corroborate how the Lipchitz conditions also arise naturally in this context.
\end{abstract}

{\bf 2000 Mathematics Subject Classification: }
%\\

\textit{Keywords:} Matrix-valued diffusions, Lipschitz conditions, Picard iterations.

%\tableofcontents

\section{Introduction}
	
Bru \cite{Bru91} introduced the so called Wishart process, which is specified by the following stochastic differential equation (SDE) valued in the space of symmetric $d\times d$ matrices for certain values $\alpha$ in the so-called Wallach set (i.e. $\alpha\in\{1,2,\ldots, d-1\}\bigcup [d-1,\infty)$) and some initial condition $X_{0}$,
\begin{equation*}\label{Wpro} 
dX_{t}=\sqrt{X_{t}}dB_{t}+dB_{t}^{T}\sqrt{X_{t}}+\alpha Idt,
\end{equation*}
where $B_t$ is a $d\times d$ matrix with each entry being a Brownian motion, all of them independent.
If $X_{t}$ belongs to the set of positive semidefinite symmetric matrices, one can properly define $\sqrt{X_{t}}$ for each $t\geq 0$, as $\sqrt{X_{t}}=U_{t}\sqrt{\Lambda_{t}}U_{t}^{T}$, where $U_{t}\Lambda_{t}U_{t}^{T}$ is the spectral decomposition of $X_{t}$, and $\sqrt{\Lambda_{t}}$ represents the diagonal matrix where the diagonal is given by the square roots of the eigenvalues increasingly ordered. 

The work in \cite{Bru91} has motivated several studies and applications, let us mention for example \cite{Katori, Mayerhofer, Konig}.
%so that previous equation makes sense. 
Under certain conditions, Bru tells us that such an equation has a unique strong solution.
To argue that, Bru \cite{Bru91} appeals first to the fact, taken from \cite{Rogers}, that the square root is an analytic function in the space of symmetric positive matrices. 
%; argument also used in other works, for example in \cite{Doumerc}. 
Then, she refers to the result in Ikeda and Watanabe \cite{Ikeda} about the unique existence of a solution of a vector-valued diffusion to conclude the uniqueness and existence of the solution. 
Here we propose taking a different route, where one needs some results on matrices.

A more general model is given by 
\begin{equation}\label{GWpro}
dX_{t}=g(X_{t})dB_{t}f(X_{t})+f(X_{t})dB_{t}^{T}g(X_{t})+b(X_{t})dt,
\end{equation}
where $g, \ f$ and $b$ are matrix valued functions acting on matrices. 
One example is when one takes $\mathbb{R}\to\mathbb{R}$ functions and uses spectral decomposition to obtain matrix-valued functions; actually, we were motivated to study equation (\ref{GWpro}) after seeing this point of view in \cite{Graczyk}. 
%Here we also are interested in the solubility of such equation and the behaviour of the eigenvalues.
%This type of extension has been treated before, for example in \cite{Mayerhofer}. 

In this paper we propose using the Picard iteration method to stablish the existence of the solution, as well as exploiting the operator character of the matrix to study the equation.
We think  that dealing directly (instead of passing through vector-valued diffusions) with the matrix-equation can be useful to obtain more insight into matrix diffusions.
Thus, we hope that our results help to complement theory already developed in papers such as 
\cite{Bru89, Bru91, Cuchiero, Donati, Graczyk2, Kang, Mayerhofer, Pfaffel, Stelzer}. 
To carry out the proof, we develop a few results suited to handle matrix-equations and which are of independent interest.

One important issue that we are not studying here is the so-called \textit{time of collision}; one might read more about this in \cite{McKean, Konig, Pfaffel, Trujillo}.

%\newpage

\section{Preliminaries}

Let $\mathcal{S}_{d\times d}$ be the set of symmetric matrices, and $\mathcal{S}_{d\times d}^{+}$ the positive semidefinite  ones. 
Let $B_{t}$ be a $d\times d$ Brownian motion (i.e. a matrix filled with independent Brownian motions).
We will focus on the following SDE valued in $\mathcal{S}_{d\times d}$:
\begin{equation*}\label{MSDE}
dX_{t}=g(X_{t})dB_{t}f(X_{t})+f(X_{t})dB_{t}^{T}g(X_{t})+b(X_{t})dt,
\end{equation*}
with initial condition $X_{0}\in \mathcal{S}_{d\times d}^{+}$, and where $g, \ f$ and $b$ are $\mathcal{S}_{d\times d}\to \mathcal{S}_{d\times d}$ functions. 
However, we will be more interested in considering $\mathbb{R}\to\mathbb{R}$ functions to construct a diffusion. 
In this case the following consideration is taken for $\mathbb{R}\to\mathbb{R}$ functions $g, \ f$ and $b$. 
If $A\in\mathcal{S}_{d\times d}$, by $g(A)$ we mean $Hg(\Lambda)H^{T}$, where $H\Lambda H^{T}$ is the spectral decomposition of $A$ and $g(\Lambda)$ is the diagonal matrix with the values $g(\lambda_{1}),\ldots,g(\lambda_{d})$ and 
$\lambda_{1}\leq \lambda_{2}\leq \ldots \leq \lambda_{d}$ are the eigenvalues of $A$ increasingly ordered. 
Under this framework it turns out that $X_{t}$ is a symmetric matrix for all $t$. 
In \cite{Horn}, Chapter 6, there is a detailed study of functions acting on spaces of matrices, an idea which is extended in functional analysis to so-called \textit{functional calculus} to define a function of an operator. 

Our aim is to investigate the condition on the functions $g,\ f$ and $b$, under which previous equation has a unique strong solution. As expected, Lipschitz conditions  will play a crucial role. Before we embark on this task, some useful results are in order.

\begin{definition}
For matrices $A$ and $B$, by 
$$A\geq 0$$ 
we mean $A$ is positive semidefinite, that is $x^{T}Ax\geq 0$ for every vector $x$, and by 
$$A\leq B$$ 
we mean $ B-A\geq 0$.
\end{definition}

\begin{remark} The following results will be useful (we used \cite{Zhang} as a general reference).\\
i) For symmetric matrices $A, B$,
\begin{equation}\label{Inq2} 
(A+B)^{2}\leq 2 A^{2}+2B^{2} .
\end{equation}
ii) For a symmetric matrix $A$ and unit vector $x$,
\begin{equation}\label{InqNice}
(x^{T}Ax)^{2}\leq x^{T}A^{2}x .
\end{equation}
A proof of (\ref{InqNice}) can be obtained using the Cauchy-Schwarz inequality. Indeed
$$(x^{T}Ax)^{2}=\langle x, Ax \rangle^{2}\leq \langle x,x \rangle \langle Ax,Ax \rangle=x^{T}A^{T}Ax=x^{T}A^{2}x.$$
\end{remark}

Next, we prove an analogous result of the Cauchy inequality, which will be useful.
\begin{proposition}\label{PropCauchy}
Let $\{A_{t},\ t\geq 0\}$ in $\mathcal{S}_{d\times d}$ with each entry being a continuous function. 
Then
\begin{equation*}
\left( x^{T}\int_{0}^{t}A_{s}ds x\right)^{2}\leq t x^{T} \int_{0}^{t}A_{s}^{2}ds x ,
\end{equation*}
for any unit vector $x$.
\end{proposition}
\begin{proof}%(\textbf{of Proposition \ref{PropCauchy}})
First fix a unit vector $x$. Now consider an equidistant partition $\{ s_{1},\ldots, s_{n}\}$ of $[0,t]$ and set $A_{i}:=A_{s_{i}}$ for $i=1,\ldots,n$. Let $\Delta>0$ be the partition size and define 
$$ F(u):= x^{T} \sum_{i=1}^{n} (\sqrt{\Delta}A_{i}-u\sqrt{\Delta}I )^{2} x.$$
We have 
\begin{equation*}
 \sum_{i=1}^{n} (\sqrt{\Delta}A_{i}-u\sqrt{\Delta}I )^{2}\\
 =\underbrace{\Delta A_{i}^{2}+\ldots+\Delta A_{n}^{2}}_{\alpha} - 2 u \underbrace{(\Delta A_{1}+\ldots+\Delta A_{n})}_{\beta}
+ u^{2} \underbrace{(\Delta+\ldots+ \Delta)}_{t}.  
\end{equation*}
Since $F(u)=x^{T}\alpha x -2u x^{T}\beta x +u^{2} x^{T} t x\geq  0$ for all $u$, then for the discriminant
$$(-2 x^{T}\beta x)^{2}- 4 x^{T}\alpha x x^{T} t x\leq 0.$$
But $x^{T} t x= t$, hence
$$x^{T}(\Delta A_{1}^{2}+\ldots+ \Delta A_{n}^{2})x t \geq ( x^{T}(\Delta A_{1}+\ldots +\Delta A_{n} )x)^{2}.$$
The result follows after taking the infinitesimal sum on both sides of the previous inequality.
\end{proof}

%\newpage

\section{Existence and uniqueness}

We will use the following criterion to establish the solubility of the stochastic equation.
\begin{definition}\label{DefLG}
Consider a function $g:\mathcal{S}_{d\times d}\to \mathcal{S}_{d\times d}$. 
We say that $g$ is \textbf{Lipschitz in matrix sense} if there exists a constant $c>0$ such that 
for any pair $A_{1},A_{2}\in\mathcal{S}_{d\times d}$ and any unit vector $x\in\mathbb{R}^{d}$ we have
\begin{equation}\label{IneqGlobalLip}
x^{T}\left( g(A_{1})-g(A_{2}) \right)^{2}x\leq c x^{T}(A_{1}-A_{2})^{2}x.
\end{equation}
\end{definition}

The following two results will be useful for Theorem \ref{ThmGlobal}, their proofs are left in the Appendix, where it is properly defined what we mean by the matrix stochastic integral. 
Notice that the next proposition resembles an isometry property, 
\begin{proposition}\label{PropIso}
Let $A_{t}$ and $C_{t},\ t\geq 0$ be matrix-valued stochastic processes such that
\begin{equation*}
\int_{0}^{t} A_{s}dB_{s}C_{s}
\end{equation*}
is well defined as an It\^{o} stochastic integral. Then, for any pair of vectors $x,y\in \mathbb{R}^{d}$ 
\begin{equation*}
E\left[y^{T}\left( \int_{0}^{t} A_{s}dB_{s}C_{s} \right)^{2}x \right]=\int_{0}^{t}E\left[x^{T}C_{s}^{T}C_{s}A_{s}A_{s}^{T}  y\right]ds.
\end{equation*}
\end{proposition}

\begin{lemma}\label{LemaBeta}
Let $\tau>0$ be fixed. For continuous adapted processes $A_{t}$ and $C_{t}$ in $\mathcal{S}_{d\times d}$, there exists $\beta>0$ such that
\begin{eqnarray*}
E\left[ x^{T}\left( \int_{0}^{t}A_{s}dB_{s}C_{s}+ \int_{0}^{t}C_{s}dB_{s}^{T}A_{s}\right)^{2}x\right]
&\leq& \beta \left| E\left[ x^{T} \left( \int_{0}^{t}A_{s}dB_{s}C_{s} \right)^{2} x\right] \right|  \\
&+& \beta \left| E\left[ x^{T} \left( \int_{0}^{t}C_{s}dB_{s}^{T}A_{s} \right)^{2} x\right] \right|,
\end{eqnarray*}
for all $t\in [0,\tau]$.
\end{lemma}

%\newpage
\begin{theorem}\label{ThmGlobal}
Suppose that $g,\ f$ and $b$ are $\mathbb{R}\to \mathbb{R}$ bounded functions 
that satisfy the property of Definition \ref{DefLG}. 
%(for example when $g:\mathbb{R}\to \mathbb{R}$ is bounded).
%Additionally, suppose also that the following two quantities are bounded by a constant $\gamma<\infty$:
%$$E\left[x^{T}b^{2}(X_{0})x\right],\  E\left[x^{T}g^{2}(X_{0})x\right],$$
%where $X_{0}\in\mathcal{S}_{d\times d}^{+}$.
Then, the stochastic differential equation
\begin{equation}\label{EqSDE}
X_{t}=X_{0}+ \int_{0}^{t}b(X_{s})ds +\int_{0}^{t}g(X_{s})dB_{s}f(X_{s})+\int_{0}^{t}f(X_{s})dB_{s}^{T}g(X_{s})
\end{equation}
has a unique strong solution in $\mathcal{S}_{d\times d}$.
\end{theorem}

The following proof follows the general structure of its vector analogue taken from \cite{Tudor}.
\begin{proof} 
Since $g, \ f$ and $b$ are bounded, there is a constant $c>0$ such that for any symmetric matrix $A$
\begin{equation}\label{IneqGlobalBound}
g(A)<cI, 
\end{equation}
and the same for $f$ and $b$.

\textbf{Uniqueness}. Let $\tau>0$ be fixed and consider $t\in [0,\tau]$. 
If $X_{t}$ and $Y_{t}$ are two solutions of the SDE (\ref{EqSDE}), we want to prove that  
$(\forall x\in\mathbb{R}^{d})\ E\left[x^{T}\left( X_{t}-Y_{t} \right)^{2}x \right]=0$.

Using inequality (\ref{Inq2}) we have
\begin{eqnarray*}
 E\left[x^{T}\left( X_{t}-Y_{t} \right)^{2}x \right] &=& E\left[x^{T}\left( \int_{0}^{t}(b(X_{s})-b(Y_{s}))ds+H_{t}(X,Y) \right)^{2}x \right]\\
 &\leq & 2 E\left[x^{T}\left(  \left(\int_{0}^{t}(b(X_{s})-b(Y_{s}))ds \right)^{2} +(H_{t}(X,Y))^{2} \right)x \right]
\end{eqnarray*}
where 
%$$H_{t}(X_{s},Y_{s}):=\int_{0}^{t}g(X_{s})dB_{s}+\int_{0}^{t}dB_{s}^{T}g(X_{s})-\int_{0}^{t}g(Y_{s})dB_{s}-\int_{0}^{t}dB_{s}^{T}g(Y_{s}).$$
$$H_{t}(X,Y):=\int_{0}^{t}g(X_{s})dB_{s}f(X_{s})+\int_{0}^{t}f(X_{s})dB_{s}^{T}g(X_{s})-\int_{0}^{t}g(Y_{s})dB_{s}f(Y_{s})-\int_{0}^{t}f(Y_{s})dB_{s}^{T}g(Y_{s})$$
With Proposition \ref{PropCauchy} and the Lipschitz condition (\ref{DefLG}) we have that
\begin{equation}\label{Inqb}
E\left[x^{T}  \left(\int_{0}^{t}(b(X_{s})-b(Y_{s}))ds \right)^{2}x\right] 
\leq \tau c\int_{0}^{t}E[x^{T} (X_{s}-Y_{s})^{2} x ]ds.
\end{equation}

%Next we do the following for the $H_{t}$-term. Using Lemma \ref{LemaBeta}, Proposition \ref{PropIso} and the Lipschitz condition we have  
%\begin{eqnarray*}
%&& E\left[x^{T}\left( H_{t}(X_{s},Y_{s}) \right)^{2}x \right]\\
%&&=      E\left[x^{T}\left( \int_{0}^{t}(g(X_{s})-g(Y_{s}))dB_{s}+\int_{0}^{t}dB_{s}^{T}(g(X_{s})-g(Y_{s})) \right)^{2}x \right]\\
%&&\leq  \beta \left| E\left[x^{T}\left( \int_{0}^{t}(g(X_{s})-g(Y_{s}))dB_{s}\right)^{2}x\right]\right| 
%+ \beta\left| E\left[x^{T}\left(\int_{0}^{t}dB_{s}^{T}(g(X_{s})-g(Y_{s})) \right)^{2}x \right]\right| \\
%&&= 2\beta \int_{0}^{t} E\left[x^{T}\left(g(X_{s})-g(Y_{s}) \right)^{2}x \right] ds\leq  2\beta \int_{0}^{t} E\left[x^{T}\left(X_{s}-Y_{s} \right)^{2}x \right] ds.
%\end{eqnarray*}

Next we do the following for the other term,
\begin{eqnarray*}
H_{t}(X,Y)&=& H_{t}(X,Y) \pm  \int_{0}^{t}g(X_{s})dB_{s}f(Y_{s})\pm \int_{0}^{t}f(Y_{s})dB_{s}^{T}g(X_{s})  \\
&=& \underbrace{\int_{0}^{t}g(X_{s})dB_{s}(f(X_{s})-f(Y_{s}))+ \int_{0}^{t}(f(X_{s})-f(Y_{s}))dB_{s}^{T}g(X_{s})}_{H^{(1)}}\\
&+& \underbrace{\int_{0}^{t}(g(X_{s})-g(Y_{s}))dB_{s}f(X_{s})+ \int_{0}^{t}f(X_{s})dB_{s}^{T}(g(X_{s})-g(Y_{s}))}_{H^{(2)}}. 
\end{eqnarray*}
Notice that last expression is the sum of two symmetric matrices, $H^{(1)}$ and $H^{(2)}$, thus, to analyze 
\begin{equation*}\label{ExpH}
E\left[x^{T} (H_{t}(X,Y))^{2}x \right],
\end{equation*}
we can apply inequality (\ref{Inq2}) to split the previous expression into two parts, one with $H^{(1)}$ and the other with $H^{(2)}$. 
After that, we can apply Lemma \ref{LemaBeta} to each part with $H^{(i)}$, so that in the end we have split it into four terms.
This means that $E\left[x^{T} (H_{t}(X,Y))^{2}x \right]$ is less than or equal to the sum of four terms, each one of the form 
$$E\left[ x^{T} \left( \int_{0}^{t}g(X_{s})dB_{s}(f(X_{s})-f(Y_{s}))\right)^{2}x\right].$$
Using Proposition \ref{PropIso}, the Lipschitz (\ref{DefLG}) and the boundedness conditions (\ref{IneqGlobalBound}), 
the following happens to each term
\begin{eqnarray*}
E\left[ x^{T} \left( \int_{0}^{t}g(X_{s})dB_{s}(f(X_{s})-f(Y_{s}))\right)^{2}x\right] &=& 
\int_{0}^{t} E\left[x^{T} (f(X_{s})-f(Y_{s}))^{2} g^{2}(X_{s}) x\right] ds \\
&\leq & c_{1} \int_{0}^{t}E[x^{T} (X_{s}-Y_{s})^{2}x]ds,
\end{eqnarray*}  
for some finite constant $c_{1}$. 
All this, together with (\ref{Inqb}), ends up giving that
\begin{equation*}
E\left[x^{T}\left( X_{t}-Y_{t} \right)^{2}x \right]\leq c_{\tau} \int_{0}^{t} E\left[x^{T}\left(X_{s}-Y_{s} \right)^{2}x \right] ds,
\end{equation*}
where $c_{\tau}$ is a constant depending on $\tau$. 
An application of Gronwall's Lemma finishes this part of the proof, which is to say that 
$E\left[x^{T}\left( X_{t}-Y_{t} \right)^{2}x \right]$ is in fact zero for all unit vectors $x$. %, therefore in norm.

\textbf{Existence}. The Picard iteration technique commands us to define
\begin{equation*}
X_{t}^{(n)}:=X_{0}+\int_{0}^{t}b(X_{s}^{(n-1)})ds+\int_{0}^{t}g(X_{s}^{(n-1)})dB_{s}f(X_{s}^{(n-1)})+\int_{0}^{t}f(X_{s}^{(n-1)})dB_{s}^{T}g(X_{s}^{(n-1)}),
\end{equation*}
and $X_{t}^{(0)}:=X_{0}$ for all $t\geq 0$. 
We want to prove that there exists a stochastic process $X_{t}$ valued in $\mathcal{S}_{d\times d}$ such that 
\begin{center}
\textbf{i)} $X_{t}^{(n)}\to X_{t}$ uniformly on $t\in[0,\tau]$  and \textbf{ii)} that $X_{t}$ satisfies the SDE (\ref{EqSDE}). 
\end{center}

First, in order to prove \textbf{i)}, with techniques already used in the Uniqueness part, i.e. inequality (\ref{Inq2}), Lemma \ref{LemaBeta}, Proposition \ref{PropIso}, as well as the boundedness condition (\ref{IneqGlobalBound}), we have 
%\begin{eqnarray*}
%&&E[x^{T}(X_{t}^{(1)}- X_{t}^{(0)})^{2}x]\\
%&& =E\left[x^{T}\left( \int_{0}^{t}b(X_{s}^{(0)})ds+\int_{0}^{t}g(X_{s}^{(0)})dB_{s}+\int_{0}^{t}dB_{s}^{T}g(X_{s}^{(0)})\right)^{2}x\right]\\
%&& \leq 2 t^{2} E[x^{T}b^{2}(X_{0})x]+ 4\beta t^{2} E[x^{T}g^{2}(X_{0})x]\leq c_{\tau},
%\end{eqnarray*}
\begin{eqnarray}\label{IneqfirstIstep}
&&E[x^{T}(X_{t}^{(1)}- X_{t}^{(0)})^{2}x]\\
&& =E\left[x^{T}\left( \int_{0}^{t}b(X_{s}^{(0)})ds+\int_{0}^{t}g(X_{s}^{(0)})dB_{s}f(X_{s}^{(0)})+\int_{0}^{t}f(X_{s}^{(0)})dB_{s}^{T}g(X_{s}^{(0)})\right)^{2}x\right]\nonumber\\
&& \leq 2 t^{2} E[x^{T}b^{2}(X_{0})x]+ 2\beta t^{2} E[x^{T}f^{2}(X_{0})g^{2}(X_{0})x]+ 2\beta t^{2} E[x^{T}g^{2}(X_{0}) f^{2}(X_{0})x]\leq c_{\tau},\nonumber
\end{eqnarray}
for all $t\in[0,\tau]$, where $c_{\tau}$ is a finite constant depending on $\tau$.

Now, using again (\ref{Inq2}), we have the following inequality
\begin{eqnarray*}
&&E[x^{T}(X_{t}^{(n+1)}- X_{t}^{(n)})^{2}x]\\
&&= E\left[x^{T} \left( \int_{0}^{t}(b(X_{s}^{(n)})-b(X_{s}^{(n-1)}))ds +H_{t}(X^{(n)},X^{(n-1)})\right)^{2} x\right]\\
&&\leq 2E\left[x^{T} \left( \int_{0}^{t}(b(X_{s}^{(n)})-b(X_{s}^{(n-1)}))ds \right)^{2} x\right] + 2E\left[x^{T} \left( H_{t}(X^{(n)},X^{(n-1)})\right)^{2} x\right].
\end{eqnarray*}

Let us analyze the last two terms in the left hand side of the previous display. 
With Proposition \ref{PropCauchy} and the Lipschitz condition (\ref{DefLG}) we obtain
\begin{equation}\label{Inqb2}
E\left[x^{T} \left( \int_{0}^{t}(b(X_{s}^{(n)})-b(X_{s}^{(n-1)}))ds \right)^{2} x\right]\leq 
\tau \int_{0}^{t} E\left[x^{T} (X_{s}^{(n)}-X_{s}^{(n-1)})^{2} x\right] ds.
\end{equation} 

For the $H$-term, using the same idea as for the Uniqueness part:
\begin{eqnarray*}
&& H_{t}(X^{(n)},X^{(n-1)})\\
&& = H_{t}(X^{(n)},X^{(n-1)})
\pm\int_{0}^{t}g(X_{s}^{(n)})dB_{s}f(X_{s}^{(n-1)})\pm\int_{0}^{t}f(X_{s}^{(n-1)})dB_{s}^{T}g(X_{s}^{(n)})\\
&& = \int_{0}^{t}g(X_{s}^{(n)})dB_{s}(f(X_{s}^{(n)})-f(X_{s}^{(n-1)}))+ \int_{0}^{t}(f(X_{s}^{(n)})-f(X_{s}^{(n-1)}))dB_{s}^{T}g(X_{s}^{(n)})\\
&& + \int_{0}^{t}(g(X_{s}^{(n)})-g(X_{s}^{(n-1)}))dB_{s}f(X_{s}^{(n)})+ \int_{0}^{t}f(X_{s}^{(n-1)})dB_{s}^{T}(g(X_{s}^{(n)})-g(X_{s}^{(n-1)})). 
\end{eqnarray*}
Using inequality (\ref{Inq2}) and iterating the same arguments (i.e. Lemma \ref{LemaBeta}, Proposition \ref{PropIso}, Lipschitz condition (\ref{DefLG}), boundedness condition (\ref{IneqGlobalBound})) as in the Uniqueness part we arrive at
%\begin{equation}
%E\left[x^{T} \left( H_{t}(X_{s}^{(n)},X_{s}^{(n-1)}) \right) x\right]\leq 
%2 \beta \int_{0}^{t} E\left[x^{T} (X_{s}^{(n)}-X_{s}^{(n-1)})^{2} x\right] ds.
%\end{equation} 
\begin{equation}\label{InqStar}
E[x^{T}(X_{t}^{(n+1)}- X_{t}^{(n)})^{2}x] \leq c_{\tau} \frac{(\beta t)^{n}}{n!},
\end{equation}
where we also used (\ref{IneqfirstIstep}) in the last iteration.
Notice that we obtain the same inequality after incorporating (\ref{Inqb2}). %which will be used subsequently.

Define now 
\begin{equation*}
D_{n}:= \sup_{t\leq \tau} \left| x^{T} (X_{t}^{(n+1)}-X_{t}^{(n)}) x\right|.
\end{equation*}
We want to prove that
\begin{equation}\label{sumfinite}
\sum_{n=1}^{\infty}P(D_{n}\geq 1/n^{2})<\infty.
\end{equation}

Using the Chebyshev inequality
$$\sum_{n=1}^{\infty}P(D_{n}\geq 1/n^{2})\leq \sum_{n=1}^{\infty}n^{4}E(D_{n}^{2}).$$
So, it suffices to show that
$$E(D_{n}^{2})\leq c_{\tau}\frac{(\beta \tau)^{n}}{n!}.$$ % to check if it is n or n-1
However,
\begin{eqnarray*}
D_{n}&=& \sup_{t\leq \tau} \left| x^{T} \left( \int_{0}^{t}(b(X_{s}^{(n)})-b(X_{s}^{(n-1)}))ds + H_{t}(X^{(n)},X^{(n-1)})\right)x\right|\\
&\leq& \int_{0}^{\tau}| x^{T} (b(X_{s}^{(n)})-b(X_{s}^{(n-1)})) x |ds + \sup_{t\leq \tau} \left| x^{T} H_{t}(X^{(n)},X^{(n-1)}) x \right|. 
\end{eqnarray*}
Thus,
\begin{equation*}
D_{n}^{2}\leq 2 \underbrace{\left(\int_{0}^{\tau} \left|x^{T} ((b(X_{s}^{(n)})-b(X_{s}^{(n-1)}))) x \right| ds \right)^{2}}_{A_{n}^{2}}+
2 \underbrace{\left( \sup_{t\leq \tau} \left| x^{T} H_{t}(X^{(n)},X^{(n-1)}) x \right|\right)^{2}}_{C_{n}^{2}}.
\end{equation*}

From the Cauchy-Schwarz inequality and using the inequality (\ref{InqNice}), we can produce
\begin{eqnarray*}
\left(\int_{0}^{\tau} \left|x^{T} ((b(X_{s}^{(n)})-b(X_{s}^{(n-1)}))) x \right| ds \right)^{2}&\leq& 
\int_{0}^{\tau}1ds \int_{0}^{\tau} \left|x^{T} ((b(X_{s}^{(n)})-b(X_{s}^{(n-1)}))) x \right|^{2}ds\\
&\leq &\tau \int_{0}^{\tau} x^{T} (b(X_{s}^{(n)})-b(X_{s}^{(n-1)})) ^{2} x ds.
\end{eqnarray*} 
Hence, for a constant $c_{\tau}^{(1)}$ depending on $\tau$, 
$$E(A_{n}^{2})\leq c_{\tau}^{(1)} \frac{(\beta\tau)^{n}}{n!}.$$

For the $H_{t}$-term, since $x^{T}H_{t}(X^{(n)}-X^{(n-1)})x$ is a martingale, by Doob's inequality
\begin{equation*}
E(C_{n}^{2})\leq 4 E[x^{T} (H_{\tau}(X^{(n)},X^{(n-1)}))^{2} x].
\end{equation*}
Upon the same argument as for (\ref{InqStar}),
$$E(C_{n}^{2})\leq c_{\tau}^{(2)} \frac{(\beta\tau)^{n}}{n!}, $$
for some finite constant $c_{\tau}^{(2)}$, therefore, for $c_{\tau}:=\max(c_{\tau}^{(1)},c_{\tau}^{(2)})$,
$$E(D_{n}^{2})\leq c_{\tau} \frac{(\beta\tau)^{n}}{n!}.$$

Since (\ref{sumfinite}) holds, by the Borel-Cantelli Lemma,
$$P(\liminf_{n}\{D_{n}<1/n^{2}\})=1.$$
This says that 
$$\text { for almost all }\omega\in \Omega\text{ there exists }N(\omega)\text{ such that }D_{n}<1/n^{2}\text{ for all }n\geq N(\omega),$$ 
which implies that $\{X_{t}^{(n)}\}$ is a Cauchy sequence a.s., since
$$x^{T}X_{t}^{(n)}x=\sum_{i=0}^{n-1}x^{T}\left( X_{t}^{(i+1)}-X_{t}^{(i)} \right)x.$$
The conclusion is that there exists a process $X_{t}$ with 
\begin{equation*}\label{ConvUni}
X_{t}^{(n)}\overset{a.s.}{\to} X_{t}\text{ uniformly in }[0,\tau], \text{ as } n\to\infty.
\end{equation*}

It remains to prove that $X_{t}$ satisfies (\ref{EqSDE}), which is achieved from the inequality
\begin{equation*}
E \left[ x^{T} \left( \int_{0}^{t}(b(X_{s}^{(n)})-b(X_{s}))ds + H_{t}(X^{(n)},X)\right)^{2}x\right]
\leq c_{\tau} \beta \int_{0}^{t}E[x^{T} \left(X_{s}^{(n)}-X_{s}\right)^{2} x]ds,
\end{equation*}
and taking $n\to\infty$. 
However, previous the inequality can be obtained repeating the same kind of arguments used along the proof.
\end{proof}

%\newpage

%\newpage

\section{Appendix}

In this secction we prove Proposition \ref{PropIso} and Lemma \ref{LemaBeta}.
In what follows, $\|A\|$ represents the operator norm and $\|A\|_{F}$ the Frobenius norm of a matrix $A$.
It is well known that the norms in the space of matrices are equivalent because it is of finite dimension.

We now give some remarks regarding the matrix stochastic integral $I:=\int_{0}^{t}A_{s}dB_{s}C_{s}$ that we use below.
Notice first that $I$ is defined as the matrix where each $(i,j)$ entry is given by
\begin{equation*}\label{IntM}
I(i,j):=\sum_{k=1}^{d}\sum_{r=1}^{d}\int_{0}^{t}A_{s}(i,k)C_{s}(r,j)dB_{s}(k,r),
\end{equation*}
where $A(i,k)$, $C(r,j)$ and $B(k,r)$ are the corresponding entries of $A$, $C$ and $B$.
Therefore, the existence of $I$ occurs if 
\begin{equation*}r
\|A(i,k)C(r,j)\|_{2}^{2}:=E\left[\int_{0}^{t}\left(A_{s}(i,k)C_{s}(r,j)\right)^{2}ds\right]<\infty,
\end{equation*}
for all $i,k,r,j\in \{1,\ldots,d\}$. 
The above expression $\|\bullet \|_{2}$ is a norm in a space of stochastic processes.

Let us see how we can construct a sequence of matrix step processes $A^{(n)}$ and $C^{(n)}$ 
such that $E[\|I^{(n)}-I\|^{2}]\to 0$, $n\to \infty$, where $I^{(n)}:=\int_{0}^{t}A^{(n)}_{s}dB_{s}C^{(n)}_{s}$.
Take precisely the two sequences such that 
$$\|A^{(n)}(i,k)C^{(n)}(r,j)-A(i,k)C(r,j) \|_{2}\to 0$$ 
as $n\to \infty$ for all $i,k,r,j\in \{1,\ldots,d\}$.
By the construction of the stochastic integral in one dimension, 
we can bound to have that $E[(I^{(n)}(i,j)-I(i,j))^{2}]\to 0$ as $n\to \infty$, 
which helps to see that $E[\|I^{(n)}-I\|^{2}_{F}]\to 0$.
Nevertheless, by the equivalence of norms $E[\|I^{(n)}-I\|^{2}]\to 0$.

%\bigskip
\subsection{Proof of Proposition \ref{PropIso}}

%\begin{proof} %(\textbf{of Proposition \ref{PropIso}})
%It can be checked the formula first for step functions, and then extend it using the convergence of the stochastic integral mentioned in the remark.http://en.wikipedia.org/wiki/Matrix_norm
\textbf{i) For step processes.}
First of all, we can check the formula for matrix step processes. 
In this case we have
\begin{eqnarray*}
E\left[y^{T}\left( \int_{0} A_{s}dB_{s}C_{s} \right)^{2}x \right]
&=&E\left[y^{T} \left( \sum_{k=0}^{n-1}A_{s_{k}}(B_{s_{k+1}}-B_{s_{k}})C_{s_{k}} \right)^{2}  x\right].
\end{eqnarray*}
Notice that when expanding the square and taking expectation, the cross terms are vanished, then we have
\begin{eqnarray*}
E\left[y^{T}\left( \int_{0} A_{s}dB_{s}C_{s} \right)^{2}x \right]
&=& \sum_{k=0}^{n-1} E\left[y^{T} \left( A_{s_{k}}(B_{s_{k+1}}-B_{s_{k}})C_{s_{k}} \right)^{2}  x\right].
\end{eqnarray*}
Thus, we have to analyze
\begin{equation*}
E\left[y^{T}  A_{s_{k}}(B_{s_{k+1}}-B_{s_{k}})C_{s_{k}} A_{s_{k}}(B_{s_{k+1}}-B_{s_{k}})C_{s_{k}}  x\right],
\end{equation*}
written in a compact form as
\begin{equation*}
E\left[a^{T} \beta c \beta b\right],
\end{equation*}
using the notation 
$$
a^{T}:=y^{T}A_{s_{k}},\ \beta:=B_{s_{k+1}}-B_{s_{k}},\ c:=C_{s_{k}} A_{s_{k}},\ b:=C_{s_{k}}  x.
$$
Notice that $\beta$ is a matrix of independent normal r.v.s with mean $0$ and variance $s_{k+1}-s_{k}$.

It will be easy to deduce the formula by analyzing the 2-dimensional case:
\begin{eqnarray*}
E\left[a^{T} \beta c \beta b\right]&=& 
E\left[
(a_{1} \ a_{2})
\left(\begin{array}{rr}
\beta_{11} & \beta_{12}\\
\beta_{21} & \beta_{22}
\end{array}\right)
\left(\begin{array}{rr}
c_{11} & c_{12}\\
c_{21} & c_{22}
\end{array}\right)
\left(\begin{array}{rr}
\beta_{11} & \beta_{12}\\
\beta_{21} & \beta_{22}
\end{array}\right)
\left(\begin{array}{r}
b_{1} \\
b_{2} 
\end{array}\right)
\right]\\
&=&%%%%%%%%%%%%%%%%%
E\left[
\left(\begin{array}{r}
a_{1}\beta_{11}+ a_{2} \beta_{21}\\
a_{1}\beta_{12}+ a_{2}\beta_{22}
\end{array}\right)^{T}
\left(\begin{array}{rr}
c_{11} & c_{12}\\
c_{21} & c_{22}
\end{array}\right)
\left(\begin{array}{r}
b_{1}\beta_{11} + b_{2}\beta_{12}\\
b_{1}\beta_{21} +b_{2} \beta_{22}
\end{array}\right)
\right]\\
&=&%%%%%%%%%%%%%%
E[a_{1}c_{11}b_{1}+a_{1}c_{21}b_{2}+a_{2}c_{12}b_{1}+a_{2}c_{22}b_{2}] (s_{k+1}-s_{k})\\
&=&%%%%%%%%%%%%%
E\left[
(b_{1}  \ b_{2})
\left(\begin{array}{rr}
c_{11} & c_{12}\\
c_{21} & c_{22}
\end{array}\right)
\left(\begin{array}{r}
a_{1} \\
a_{2} 
\end{array}\right)
\right](s_{k+1}-s_{k})\\
&=&%%%%%%%%%%
E[b^{T}ca](s_{k+1}-s_{k})=E(x^{T}C^{T}CAA^{T}y)(s_{k+1}-s_{k}).
\end{eqnarray*}
This helps to see how the formula arises for step processes.

\textbf{ii) For more general processes.} 

%Recall that the stochastic integral $\int A_{s}dB_{s}C_{s}$ is constructed by taking limits of 
%$\int A_{s}^{(s)}dB_{s}C_{s}^{(n)}$ in the corresponding Hilbert space, 
%where $A_{t}^{(n)}$ and $C_{s}^{(n)}$ respectively approximate $A_{t}$ and $C_{t}$, in the sense that 
%\begin{equation*}%\label{ConvStep}
%E\left[  \int_{0}^{t} \left\| A_{s}^{(n)} - A_{s} \right\|^{2}ds  \right]\to 0,\ n\to\infty.
%\end{equation*}   

%; in which case we have that
%\begin{equation*}
%E\left[ y^{T}\left( \int_{0}^{t}A_{s}^{(n)}dB_{s}C_{s}^{(n)} - \int_{0}^{t}A_{s}dB_{s}C_{s}\right)^{2}x\right]\to 0,\ n\to\infty,
%\end{equation*}
%for any pair of vectors $x,y$.

%Denote
%\begin{equation*}
%I(n):=\int_{0}^{t}A_{s}^{(n)}dB_{s}C_{s}^{(n)}\text{ and } I:=\int_{0}^{t}A_{s}dB_{s}C_{s}.
%\end{equation*}

Let $A$ and $C$ be matrix stochastic processes where the stochastic integral $I$ is well defined.
Therefore, as mentioned above, there are approximating step processes $A^{(n)}$ and $C^{(n)}$ whose stochatic integral 
$I^{(n)}$ converges to $I$ in the $L_{2}$-norm. 

By point \textbf{i)} above,
\begin{equation*}
E\left[ y^{T}(I^{(n)})^{2}x\right]
=\int_{0}^{t}E\left[x^{T}(C_{s}^{(n)})^{T}C_{s}^{(n)}A_{s}^{(n)}(A_{s}^{(n)})^{T}  y\right]ds
\end{equation*}
for every $n \geq 1$.
Then, we want to prove that 
\begin{equation}\label{wanting1}
\left|E\left[ y^{T} (I^{(n)})^{2} x\right] - E\left[y^{T} I^{2} x\right]\right|\to 0,
\ n\to\infty,
\end{equation}
and that 
\begin{equation}\label{wanting2}
\int_{0}^{t}E\left[x^{T}(C_{s}^{(n)})^{T}C_{s}^{(n)}A_{s}^{(n)}(A_{s}^{(n)})^{T}  y\right]ds\to \int_{0}^{t}E\left[x^{T}C_{s}^{T}C_{s}A_{s}A_{s}^{T}  y\right]ds,\ n\to\infty.
\end{equation}

For (\ref{wanting1}) we have
\begin{eqnarray*}
\left|E\left[y^{T}(I(n)^{2}-I^{2})x\right] \right| &=& E\left[\left| y^{T}((I(n)-I)I(n) +I (I(n)-I)) x \right| \right]\\
&\leq & E\left[\left| y^{T}(I(n)-I)I(n) x \right| \right]   +E\left[\left| y^{T} I (I(n)-I) x \right| \right]\\
&\leq & \|x\| \|y\| \left( E[\|I(n)-I\| \|I(n)\|]+E[\|I\|  \|I(n)-I\|]\right)\\
&\leq & \|x\| \|y\| \left( \sqrt{E[ \|I(n)\|^{2}]}+\sqrt{E[\|I\|^{2} ]}\right)\sqrt{E[\|I(n)-I\|^{2}]}.
\end{eqnarray*}
Hence, we obtain (\ref{wanting1}), because $E[\|I(n)-I\|^{2}]\to 0 $ as $n\to \infty$.

For (\ref{wanting2}), we need to calculate 
$$
\int_{0}^{t}
E\left[x^{T}\left((C_{s}^{(n)})^{T}C_{s}^{(n)}A_{s}^{(n)}(A_{s}^{(n)})^{T}-C_{s}^{T}C_{s}A_{s}A_{s}^{T}\right) y\right]ds.
$$
Observe that we need to calculate 
$$E[\int_{0}^{t}(a^{n}_{1}(s)a^{n}_{2}(s)c^{n}_{1}(s)c^{n}_{2}(s)-a_{1}(s)a_{2}(s)c_{1}(s)c_{2}(s))ds],$$
where $a^{n}_{1}(s)$ and $a^{n}_{2}(s)$ are arbitrary entries of $A_{s}^{(n)}$, $c^{n}_{1}(s)$ and 
$c^{n}_{2}(s)$ of $C_{s}^{(n)}$, and similarly without the the index $n$, i.e. $a_{1}(s)$ represents an entry of $A_{s}$.

After adding and substracting $a_{1}(s)a^{n}_{2}(s)c_{1}(s)c^{n}_{2}(s)$ we can split into two terms.
Let us elaborate one of them, the other one is similar.
We have that
\begin{eqnarray*}
&& E\left[\int_{0}^{t}(a^{n}_{1}(s)a^{n}_{2}(s)c^{n}_{1}(s)c^{n}_{2}(s)-a_{1}(s)a^{n}_{2}(s)c_{1}(s)c^{n}_{2}(s))ds\right]\\
&& =E\left[\int_{0}^{t}(a^{n}_{1}(s)c^{n}_{1}(s)-a_{1}(s)c_{1}(s)) a^{n}_{2}(s) c^{n}_{2}(s) ds\right]\\
&& \leq E\left[\sqrt{\int_{0}^{t}(a^{n}_{1}(s)c^{n}_{1}(s)-a_{1}(s)c_{1}(s))^{2}ds }
\sqrt{\int_{0}^{t}(a^{n}_{2}(s)c^{n}_{2}(s))^{2} ds}\right]\\
&& \leq \sqrt{E\left[\int_{0}^{t}(a^{n}_{1}(s)c^{n}_{1}(s)-a_{1}(s)c_{1}(s))^{2}ds \right] 
E\left[\int_{0}^{t}(a^{n}_{2}(s)c^{n}_{2}(s))^{2} ds \right]}.
\end{eqnarray*}
where we used the Cauchy-Schwarz inequality twice, one for the integral and another one for the expectation.
Since $E[\int_{0}^{t}(a^{n}_{1}(s)c^{n}_{1}(s)-a_{1}(s)c_{1}(s))^{2}ds ]$ vanishes as $n\to\infty$, 
we obtain (\ref{wanting2}), and therefore the result.
%Using (\ref{ConvStep}) we obtain (\ref{wanting2}). 

%\end{proof}

%\bigskip %%%%%%%%%%%%%%%%%%%%%%%%%%%%%%%%%%%%%%%%%%%%%%%%%%%%%%%%%%%%%%%%%%%%
\subsection{Proof of Lemma \ref{LemaBeta}}

%\begin{proof}%(\textbf{of Lemma \ref{LemaBeta}})
Define $M_{t}:=\int_{0}^{t}A_{s}dB_{s}C_{s}$. Since $M_{t}+M_{t}^{T}$ is symmetric, $(M_{t}+M_{t}^{T})^{2}$ is positive semidefinite, that is
$$0\leq (M_{t}+M_{t}^{T})^{2}=M_{t}^{2}+(M_{t}^{T})^{2}+M_{t}M_{t}^{T}+ M_{t}^{T}M_{t},$$
then
$$-(M_{t}M_{t}^{T}+ M_{t}^{T}M_{t})\leq M_{t}^{2}+(M_{t}^{T})^{2}.$$
So that
\begin{eqnarray*}
-E\left[ x^{T}(M_{t}M_{t}^{T}+ M_{t}^{T}M_{t})x\right]
&\leq& E[x^{T}M_{t}^{2}x]+ E[x^{T}(M_{t}^{T})^{2}x]\\
&\leq& \left| E[x^{T}M_{t}^{2}x] \right|+ \left| E[x^{T}(M_{t}^{T})^{2}x]\right|.
\end{eqnarray*}
Now, for each $t\in[0,\tau]$, we can find $\alpha_{t}>0$ such that
\begin{eqnarray*}
\alpha_{t}\left| E\left[ x^{T}(M_{t}M_{t}^{T}+ M_{t}^{T}M_{t})x\right]\right|
&\leq& \left| E[x^{T}M_{t}^{2}x] \right|+ \left| E[x^{T}(M_{t}^{T})^{2}x]\right|.
\end{eqnarray*}
From the continuous trajectories of $B_{t}$, we have that $\alpha:[0,\tau]\to(0,\infty)$ is actually continuous. Let us then define $\delta:=\min_{t\in[0,\tau]}\alpha_{t}$.
Then
\begin{eqnarray*}
E\left[ x^{T}(M_{t}+M_{t}^{T})^{2}x\right]
&=& E[x^{T}M_{t}^{2}x]+E[x^{T}(M_{t}^{T})^{2}x]\\
&& +E[x^{T}(M_{t}M_{t}^{T}+M_{t}^{T}M_{t})x]\\
&\leq & \left| E[x^{T}M_{t}^{2}x] \right| +\left| E[x^{T}(M_{t}^{T})^{2}x]\right| \\ 
&&+  \frac{1}{\delta} \left\{ \left| E[x^{T}M_{t}^{2}x]\right| +\left| E[x^{T}(M_{t}^{T})^{2}x]\right| \right\}.
\end{eqnarray*}
Defining $\beta:=1+\delta^{-1}$ gives the inequality.
%\end{proof}

%\bigskip %%%%%%%%%%%%%%%%%%%%%%%%%%%%%%%%%%%%%%%%%%%%%%%%%%%%%%%%%%%%%%%%%%%%

%\subsection{Proof of Theorem \ref{ThmHorn}}

%\bigskip %%%%%%%%%%%%%%%%%%%%%%%%%%%%%%%%%%%%%%%%%%%%%%%%%%%%%%%%%%%%%%%%%%%%

%\subsection{Gronwall's lemma}

%\newpage

\end{document}